\documentclass[11pt]{amsart}
\usepackage{mathrsfs}
\parskip=\smallskipamount

\newtheorem{theorem}{Theorem}[section]

\newtheorem{lemma}[theorem]{Lemma}
\newtheorem{corollary}[theorem]{Corollary}
\newtheorem{proposition}[theorem]{Proposition}

\theoremstyle{definition}
\newtheorem{definition}[theorem]{Definition}

\newcommand{\cA}{\mathcal{A}}

\newcommand{\cR}{\mathcal{R}}
\newcommand{\cC}{\mathcal{C}}

\newcommand{\cF}{\mathcal{F}}

\newcommand{\B}{\mathbb{B}}
\newcommand{\C}{\mathbb{C}}
\newcommand{\CC}{\mathbb{C\/}^2}
\newcommand{\D}{\mathbb{D}}

\newcommand{\bL}{\mathbb{L}}
\newcommand{\N}{\mathbb{N}}

\newcommand{\R}{\mathbb{R}}
\newcommand{\RN}{\mathbb{R\/}^N}

\newcommand{\debar}{\bar{\partial}}

\newcommand\Subset{\subset\subset}

\newcommand{\row}[2]{#1_1,\ldots,#1_#2}
\newcommand{\oO}{\overline\Omega}
\newcommand{\jN}{j=1,\ldots,N}

\newcommand{\jn}{j=1,\ldots,n}

\def\tphi{\tilde \Phi}

\numberwithin{equation}{section}

\begin{document}
\title[Analytic Structure]{Presence or absence of analytic structure in maximal ideal spaces}
\author[A.\ J.\ Izzo]{Alexander J. Izzo}
\thanks{The first author was partially supported by NFR grant 209751/F20.}
\address{Alexander J. Izzo, Department of Mathematics and Statistics, Bowling Green State University, Bowling Green, OH, 43403 USA}
\email{aizzo@bgsu.edu}
\author[H.\ Samuelsson Kalm]{H\aa kan Samuelsson Kalm}
\thanks{The second author was partially supported by the Swedish Research Council.}
\address{H\aa kan Samuelsson Kalm, Mathematical Sciences, Chalmers University of Technology and the University of Gothenburg,
SE-412 96 G\"{o}teborg, Sweden}
\email{hasam@chalmers.se}
\author[E.\ Forn\ae ss Wold]{Erlend Forn\ae ss Wold}
\thanks{The third author was supported by NFR grant 209751/F20.}
\address{Erlend Forn\ae ss Wold,  Matematisk Institutt, Universitetet i Oslo,
Postboks 1053 Blindern, 0316 Oslo, Norway} 
\email{erlendfw@math.uio.no}

%
%

\begin{abstract}
We study extensions of Wermer's maximality theorem to several complex variables.
We exhibit various smoothly embedded manifolds in complex Euclidean space whose hulls are non-trivial but contain no analytic disks.  We answer a question posed by Lee Stout concerning the existence of analytic structure for a uniform algebra whose maximal ideal space is a manifold.
\end{abstract}

\maketitle

\section{introduction}

A central theme in the theory uniform algebras is to find analytic structure in the maximal ideal space of a given algebra.  
For $M$ a compact space and $f_1,\ldots, f_N$ continuous complex-valued 
functions on $M$, we will denote by $[f_1,\ldots, f_N]_M$ the uniform algebra on $M$ generated by $f_1,\ldots, f_N$.  Setting $f=(f_1,\ldots,f_N)\colon M\to \C^N$, the uniform algebra $[f_1,\ldots, f_N]_M$ is isomorphic to $[z_1,\ldots,z_N]_{f(M)}$, \emph{i.e.}, to the uniform algebra on $f(M)\subset \C^N$
generated by the complex coordinate functions. Hence, the maximal ideal space of $[f_1,\ldots, f_N]_M$ is isomorphic to the polynomially convex hull
$\widehat{f(M)}$ of $f(M)$. In particular, notice that analytic structure in $\widehat{f(M)}$ will prevent $[f_1,\ldots, f_N]_M$ from being 
the algebra $\cC(M)$ of all continuous functions on $M$, and that the maximal ideal space of $[f_1,\ldots, f_N]_M$ is $M$ if and only if $f(M)$ is 
polynomially convex. 
In this paper we will mainly be concerned with the case when $M$ is the boundary of some domain in $\C^n$
with polynomially convex closure.
For notational convenience, we will denote $[z_1,\ldots,z_n, f_1,\ldots, f_N]$ by $[z,f]$, and we will denote the graph of $f$ over $M$ (\emph{i.e.}, 
the image of $M$ under the map $(z,f)$) by $\mathcal G_f(M)$.
Throughout the paper the word \lq\lq smooth\rq\rq\ will mean of class $\cC^\infty$ except where explicitly indicated otherwise.

\smallskip

Recall Wermer's maximality theorem \cite{Wermer53}: 
If $f$ is a continuous function on the unit circle $b\mathbb D\subset \C$, then either $f$ is the boundary 
value of a holomorphic function or else the uniform algebra $[z,f]_{b\mathbb D}$ generated by $z$ and $f$ is 
equal to $\cC(b\mathbb D)$.   Since $[z,\frac{1}{z}]_{b\mathbb D}=\cC(b\mathbb D)$ it is clear from 
the Oka-Weil theorem that $[z,f]_{b\mathbb D}=\cC(b\mathbb D)$ if and only if the graph $\mathcal G_f(b\mathbb D)$ of $f$
is polynomially convex.  Thus for a continuous function $f$ on $b\mathbb D$ the following four conditions are equivalent:
\begin{itemize}
\item[(1)] $\mathcal G_f(b\D)$ is polynomially convex.
\item[(2)] $\widehat{\mathcal G_f(b\D)}\setminus \mathcal G_f(b\D)$ contains no analytic disk.
\item[(3)] $f$ does not extend continuously to a holomorphic function on the unit disk $\D$.  
\item[(4)] $[z,f]_{b\D}=\cC(b\D)$.
\end{itemize}
We sketch here a proof of $(3) \implies (1)$ that illustrates our approach in this paper.
Let $\tilde f$ be the harmonic extension of $f$ to $\mathbb D$.   
Using harmonic conjugates and the fact that $\mathbb{D}$ is starshaped, 
it is not hard to see that $\mathcal G_{\tilde f}(\overline{\mathbb D})$ is polynomially convex.  
Assume that $\tilde{f}$ is not holomorphic; then the set $A\subset \mathbb{D}$ of points where $\overline\partial\tilde f = 0$
is discrete. One can show that every point in $\mathbb{D}\setminus A$ is a local peak point for the algebra
$[z,\tilde f]_{\overline{\mathbb D}}$. It then follows from Rossi's local maximum principle (see, \emph{e.g.}, \cite{Rosay2006} or \cite[Theorem~III.8.2]{Gamelin})
that $(z,f(z))\notin \widehat{\mathcal{G}_f(b\mathbb{D})}$ for every $z\in \mathbb{D}\setminus A$.
Since $A$ is discrete it follows that $\mathcal{G}_f(b\mathbb{D})$ is polynomially convex.

One could ask about the possibility of carrying over the equivalence of (1), (2), (3), and (4) above to the setting of several complex variables.  Specifically, one could ask whether for $\Omega\subset \C^n$ a sufficiently nice domain and $f_1,\ldots, f_n$ continuous functions on $\cC(b\Omega)$ the following four conditions are equivalent:
\begin{itemize}
\item[(1)] $\mathcal G_f(b\Omega)$ is polynomially convex.
\item[(2)]  $\widehat{\mathcal G_f(b\Omega)}\setminus \mathcal G_f(b\Omega)$ contains no analytic disk.
\item[(3)] There does not exist an analytic set\footnote{Throughout the paper, by an \emph{analytic set} we mean a subset of $\C^n$ that is locally the common zero set of finitely many holomorphic functions.  Such sets are often referred to as analytic varieties or holomorphic varieties.} ``attached'' to $b\Omega$ to which $f$ extends continuously as a holomorphic map.
\item[(4)] $[z,f]_{b\Omega}=\cC(b\Omega)$.
\end{itemize}
Of course it is always true that $(4) \implies (1) \implies (2)\implies (3)$.  That these implications are not reversible is shown in the work of Richard Basener \cite{Basener}.
Specifically, letting $\B_n$ denote the unit ball in $\C^n$, Basener showed that there exist smooth functions $f_1,\ldots, f_4$ on 
$S^3=b\B_2\subset\C^2$ such that $\mathcal G_f(b\B_2)$ is polynomially convex but $[z,f]_{b\B_2}\neq \cC(b\B_2)$, and he also observed that there exist different smooth functions 
$f_1,\ldots, f_3$ such that 
the polynomially convex hull of $\mathcal G_f(b\B_2)\subset\mathbb C^5$ 
is non-trivial but contains no analytic disk.  Theorems~\ref{graphtorus},~\ref{graphball}, and~\ref{embedM} exhibit further instances of this phenomenon of smooth manifolds in $\C^n$ with non-trivial polynomially convex hull without analytic structure, and in particular, we strengthen the second result of Basener by reducing the number of functions needed.  

In view of the results of Basener, we need some conditions on the $f_j$'s if we are to obtain multivariable versions of Wermer's maximality theorem. 
A difference 
between one and several complex variables is that in $\mathbb C^n$ for $n\geq 2$, the Dirichlet problem, while solvable for harmonic functions, is not in general solvable with {\em pluriharmonic} functions, and it is only the pluriharmonic functions that have conjugates.  In seeking extensions of Wermer's maximality theorem to several complex variables, it is therefore natural to restrict consideration to those functions that are boundary values of pluriharmonic functions.  
Uniform algebras generated by holomorphic and pluriharmonic functions of several complex variables were studied by E.~M.~\v Cirka in \cite{Cirka}, the first author in \cite{Izzo93} and \cite{Izzo95}, and the second and third authors in \cite{SW}.  
We will denote the set of all complex-valued pluriharmonic functions on a domain $\Omega$ by $PH(\Omega)$.
Our first result is closely related to a result in \cite{SW}.

\begin{theorem}\label{main1}
Let   $\Omega\subset \C^n$ be a bounded domain with $\mathcal C^1$-smooth boundary and polynomially convex closure, and let $h_j\in PH(\Omega)\cap\cC(\overline\Omega)$ for $\jN$.  Then the following four conditions are equivalent:
\begin{itemize}
\item[(1)] $\mathcal G_h(b\Omega)$ is polynomially convex.
\item[(2)]  $\widehat{\mathcal G_h(b\Omega)}\setminus \mathcal G_h(b\Omega)$ contains no analytic disk.
\item[(3)] There does not exist a nontrivial analytic disk $\triangle\hookrightarrow\Omega$ on which all the $h_j$'s are holomorphic.  
\item[(4)] $[z,h]_{\oO}=\{\, f\in \mathcal{C}(\oO):f|_{b\Omega}\in [z,h]_{b\Omega}\,\}$.
\end{itemize}
\end{theorem}

As an immediate corollary we have the following.

\begin{corollary}\label{cor1}
Let $\Omega$ and $\row hN$  be as in Theorem~\ref{main1}.  Then
\begin{equation}\label{super}
[z,h]_{b\Omega}=\cC(b\Omega) \quad \Longrightarrow \quad [z,h]_{\overline\Omega}=\cC(\overline\Omega).
\end{equation}
\end{corollary}

This result sharpens Theorem~1.3 in \cite{SW}, where it is proved that if $[z,h]_{b\Omega}=\cC(b\Omega)$, then
either there exists an analytic disk in $\Omega$ on which each $h_j$ is holomorphic, 
or else $[z,h]_{\overline\Omega}=\cC(\overline\Omega)$.
A related version of ``the maximality theorem'' for the distinguished boundary $\Gamma$ of the bi-disk $\mathbb D^2$
is also proved in \cite{SW}: 
Let $h_j\in PH(\mathbb D^2)\cap\cC(\overline{\mathbb D}^2)$.   Then either there exists an algebraic 
variety $Z\subset\mathbb C^2$ with $Z\cap b\mathbb D^2\subset\Gamma$ on which all $h_j$ are holomorphic, 
or $[z,h]_{\Gamma}=\cC(\Gamma)$.
In particular, starting with continuous functions $f_j$ on $\Gamma$, the only obstruction to $[z,f]_{\Gamma}=\cC(\Gamma)$ is 
the presence of analytic structure in the maximal ideal space of the algebra $[z,f]_{\overline{\mathbb{D}}^2}$ 
\emph{provided} we assume that the $f_j$'s extend to pluriharmonic functions on the bi-disk.  Theorem~\ref{graphtorus} below shows that this becomes false if the plurihamonicity condition is dropped.

Theorem~\ref{main1} implies another extension of Wermer's maximality theorem to several complex variables.  A result of Lee Stout \cite{Stout1.5} asserts that  the complex polynomials are uniformly dense in 
the continuous functions on any compact, polynomially convex, real-analytic subvariety of
complex Euclidean space.
In combination with Theorem~\ref{main1} this gives that in the case when the $h_j$'s are \emph{real-analytic} on $b\Omega$ we can replace condition (4) in Theorem~\ref{main1} by the condition that $[z,h]_{\overline\Omega}=\cC(\overline\Omega)$.  Explicitly we obtain the following result.\footnote{When every point of $b\Omega$ is known to be a peak point 
for $A(\Omega)$ [e.g., for strictly pseudoconvex domains] 
one can invoke a weaker result of Anderson, Izzo, and Wermer \cite{AIW} in place of Stout's theorem.}

\begin{corollary}\label{cor2}
Let $\Omega$ and $\row hN$  be as in Theorem~\ref{main1}.  Suppose in addition that $b\Omega$ is real-analytic and the $\row hN$ are 
real-analytic on $b\Omega$.  Then the following four conditions are equivalent:
\begin{itemize}
\item[(1)] $\mathcal G_h(b\Omega)$ is polynomially convex.
\item[(2)]  $\widehat{\mathcal G_h(b\Omega)}\setminus \mathcal G_h(b\Omega)$ contains no analytic disk.
\item[(3)] There does not exist a nontrivial analytic disk $\triangle\hookrightarrow\Omega$ on which all the $h_j$'s are holomorphic.  
\item[(4)] $[z,h]_{\overline\Omega}=\cC(\overline\Omega)$.  
\end{itemize}
\end{corollary}

Corollary~\ref{cor2} becomes false, in general, if the real-analyticity hypotheses are dropped.  On the (nonsmooth) bi-disk $\mathbb D^2$ with $h_1=\overline{z_1-1}$ and $h_2=\overline{z_2(z_1-1)}$, condition (3) is satisfied while condition (4) fails since $h_1$ and $h_2$ are both identically zero on the analytic disk $\{z_1=1, |z_2|\leq 1\}$ lying in $b\mathbb D^2$.  Replacing the bi-disk by a smoothly bounded domain $\Omega$ such that $\{z_1=1, |z_2|\leq 1\}\subset b\Omega$ and $\{z_1=1\}\cap \Omega$ is empty gives a counterexample to Corollary~\ref{cor2} without the real-analyticity hypotheses.   
We conjecture, however, that for \emph{strictly pseudoconvex} domains, Corollary~\ref{cor2} remains true without the real-analyticity hypotheses.  Note that this conjecture constitues an $n$-dimensional generalization of Wermer's maximality theorem.

When $\mathcal G_h(b\Omega)$ fails to be polynomially convex we can say more than just that there is an analytic disk in $\widehat{\mathcal G_h(b\Omega)}\setminus \mathcal G_h(b\Omega)$; we show that there is an analytic set $Z\subset \Omega$ that can be foliated in such a way the $h_j$'s are holomorphic along the plaques, and the graph of $h$ over $Z$ is contained in $\widehat{\mathcal G_h(b\Omega)}\setminus \mathcal G_h(b\Omega)$.  However, this \lq\lq foliation\rq\rq\ can have singularities, so it is not a regular foliation in the usual sense.  We therefore make the following definition.

\begin{definition}\label{singularfol}
A {\em singular foliation $\{ (U_\alpha, F_\alpha)\}$ of an analytic set $Z$ by non-trivial varieties\/} is a cover of $Z$ by open sets $U_\alpha$ together with holomorphic maps $F_\alpha$ on $U_\alpha$ such that each level set of each $F_\alpha$ has dimension $\geq 1$ and in each nonempty intersection $U_\alpha\cap U_\beta$ the collection of components of the level sets of $F_\alpha$ coincides with the collection of components of the level sets of $F_\beta$.  For each $\alpha$, the level sets of $F_\alpha$ are called {\em plaques}.
\end{definition}

Given a singular foliation of an analytic set $Z$, it is possible to piece together the plaques to obtain a partition of $Z$ into disjoint \lq\lq leaves\rq\rq\ that are \lq\lq immersed analytic sets\rq\rq, in a manner similar to how, given a regular foliation, one obtains leaves that are immersed submanifolds.  Since we will not need the leaves of a singular foliation, we omit the proof.

\begin{theorem}\label{foliation}
Let $\Omega$ and $\row hN$  be as in Theorem~\ref{main1}. Suppose that $\mathcal G_h(b\Omega)$ is not polynomially convex.  Then there is an analytic set $Z\subset \Omega$
of dimension $d\geq 1$ and a singular foliation $\mathcal{F}$ of $Z$ by non-trivial varieties such that
all the $h_j$ are strongly holomorphic along the plaques of $\mathcal{F}$.  The graph  
$\mathcal G_h(Z)$ is contained in $\widehat{\mathcal G_h(b\Omega)}\setminus \mathcal G_h(b\Omega)$. 
\end{theorem}

As mentioned above, without the assumption that our functions $f_j\colon b\Omega \to \C$ have pluriharmonic 
extensions to $\Omega$, the presence or absence of analytic structure in the maximal ideal space 
of $[z,f]_{b\Omega}$ is more delicate.
Our next three results show that it is not uncommon that the maximal ideal space of $[z,f]_{b\Omega}$
strictly contains $b\Omega$ but lacks analytic structure.  As mentioned above the first of these results shows that pluriharmonicity cannot be omitted in \cite[Theorem~1.3]{SW}, and the second improves a result of Basener by decreasing the dimension of the ambient space.  The third shows that every smooth manifold of dimension at least three smoothly embeds in some complex Euclidean space so as to have a non-trivial hull without analytic structure.

\begin{theorem}\label{graphtorus}
There exists a real-valued smooth function $f$ on $\Gamma=(b\D)^2\subset\C^2$ such that 
the polynomially convex hull of $\mathcal G_f(\Gamma)\subset\mathbb C^3$ 
is non-trivial but contains no analytic disk.  
\end{theorem}

\begin{theorem}\label{graphball}
There exist real-valued smooth functions $f_1$ and $f_2$ on $S^3=b\B_2\subset\C^2$ such that 
the polynomially convex hull of $\mathcal G_f(S^3)\subset\mathbb C^4$ 
is non-trivial but contains no analytic disk.  
\end{theorem}

\begin{theorem}\label{embedM}
If $M$ is a smooth compact manifold-with-boundary of real dimension $m\geq 3$, then there is a smooth embedding
$F\colon M\to \C^{2m+4}$ such that the polynomially convex hull of $F(M)$ is nontrivial but contains no analytic disk. 
\end{theorem}

Even in dimension 1, it is trivial that Corollary~\ref{cor1} fails without the pluriharmonicity hypothesis since $h$ could be holomorphic on a nonempty proper open subset of $\Omega$ while not agreeing with the boundary values of a holomorphic function on $b\Omega$.  The next result provides a more interesting illustration of what can go wrong in Corollary~\ref{cor1}, even in dimension~1, without the pluriharmonicity hypothesis.  (Here we denote by $\cR(K)$ the uniform closure on $K$ of the rational functions holomorphic on a neighborhood of $K$.)

\begin{theorem}\label{smoothexample1}
Let   $\Omega\subset \C^1$ be a bounded open set.  There exist functions $f_1,f_2,f_3\in \cC^\infty(\oO)$ and a compact set $K\subset\Omega$ such that the following hold:
\begin{itemize}
\item[(1)] $\mathcal G_f(\oO)$ is polynomially convex, 
\item[(2)]  the Shilov boundary of $[z,f]_{\overline\Omega}$ is $\oO$, so in particular there is no disk where all the $f_j$'s are holomorphic,
\item[(3)] $\cC(\overline\Omega)\cap\mathcal \cR(K)\subset[z,f]_{\overline\Omega}$, and
\item[(4)] $[z,f]_{\overline\Omega}\neq\cC(\overline\Omega)$.
\end{itemize}
\end{theorem}

A theorem of John Anderson and the first author \cite{AI1} shows that it is not possible to strengthen condition (2) above to require that every point of $\oO$ be a peak point for $[z,f]_{\overline\Omega}$.  However, in $\C^n$, for $n\geq 2$, this stronger condition can be achieved as well.

\begin{theorem}\label{smoothexample2}
Let   $\Omega\subset \C^n$, with $n\geq 2$, be a bounded open set.  There exist functions $\row fN\in \cC^\infty(\oO)$ and a compact set $K\subset\Omega$ such that  the conditions in Theorem~\ref{smoothexample1} are satisfied with condition (2) replaced by:
\begin{itemize}
\item[($2'$)]  every point of $\oO$ is a peak point for $[z,f]_{\overline\Omega}$, so in particular there is no analytic disk in $\Omega$ on which all the $f_j$'s are holomorphic.
\end{itemize}
\end{theorem}

Although a nontrivial polynomially convex hull need not contain analytic structure, one could ask whether the presence of a (smooth) manifold of dimension at least 2 in $\hat X\setminus X$ ($X$ a compact set in $\C^n$) implies the existence of analytic structure.  The answer is no.  In fact, given positive integers $k<n$, there is a compact set $X\subset b\B_n\subset \C^n$ such that $\hat X\setminus X$ contains a smooth $k$-manifold but contains no analytic disk.  This follows immediately from the following result of Julien~Duval and Norman~Levenberg by taking the set $K$ there to be defined by $K=\{ (x_1,\ldots, x_k,0,\ldots,0)\in \R^n\subset\C^n: \sum_{j=1}^k |x_j|^2\leq 1/2\}$.

\begin{theorem}[Duval--Levenberg \cite {DL}] \label{duval}
If $K$ is a compact, polynomially convex subset of the ball $\B_n\subset\C^n$, $n\geq 2$, then there is a compact subset $X$ of $b\B_n$ such that $\hat X\supset K$ and such that the set $\hat X\setminus K$ contains no analytic disk.
\end{theorem}

However, we show that if there is an {\em open\/} subset of the hull $\hat X\setminus X$ that is a smooth manifold, then there necessarily is analytic structure in the hull.   More precisely we have the following result.

\begin{theorem}\label{astructureinhull}
Let $X$ be a compact set in $\C^n$.  Suppose $U$ is an open subset of $\hat X\setminus X$ and is a $\cC^1$-smooth submanifold of $\C^n$ of dimension at least 2.  Then there exists a dense open subset $\Omega$ of $U$ each component of which is an integrable CR-submanifold of $\C^n$.
\end{theorem}

More generally we have the following result about uniform algebras.

\begin{theorem}\label{stoutgen}
Let $\cA$ be a uniform algebra.  Suppose that $U$ is an open subset of the maximal ideal space of $\cA$ disjoint from the Shilov boundary of $\cA$ and that $U$ is a manifold of real dimension $n\geq 2$ with a differentiable structure such that the collection of functions in $\cA$ that are $\cC^1$ on $U$ is dense in $\cA$.  Then there exists a dense open subset $\Omega$ of $U$ each component of which has an integrable CR-structure $\cF$ such that every function in $\cA$ is holomorphic along the leaves of $\cF$.
\end{theorem}

In \cite{Stout} Stout studied uniform algebras whose maximal ideal spaces are $\cC^1$-smooth surfaces and which admit sets of $\cC^1$-smooth generators, and he showed that such algebras consist of functions holomorphic off their Shilov boundaries.  In connection with this result he raised the question whether there is also always analytic structure when the maximal ideal space is a manifold of higher dimension \cite[Section~5, Question~2]{Stout2}.  The above result answers this question of Stout in the affirmative.  Specifically specializing Theorem~\ref{stoutgen} to the case in which the maximal ideal space is a manifold gives the following.

\begin{theorem}\label{stout}
Let $M$ be a compact $\mathcal C^1$-smooth manifold (possibly with boundary) of 
real dimension $n\geq 2$, and let $\mathcal A$ be a uniform 
algebra on $M$ generated by $\mathcal C^1$-smooth functions. 
Assume further that the maximal ideal space of $\mathcal A$ is $M$, and 
let $\Gamma_\cA$ denote the Shilov boundary of $\mathcal A$.  Then there exists
a dense open subset $\Omega$ of $M\setminus\Gamma_\cA$ each component of which has an integrable 
CR-structure $\mathcal F$ such that every function in $\mathcal A$
is holomorphic along the leaves of $\mathcal{F}$.
\end{theorem}

We remark that the dense open set $\Omega$ in Theorems~\ref{astructureinhull}--\ref{stout} may have several connected components 
and the CR-dimension of different components may be different.


\section{Proof of Theorems \ref{main1} and \ref{foliation}}

We will prove the equivalence of conditions (1), (3), and (4) in Theorem~\ref{main1} first.
Theorem~\ref{foliation} will follow essentially as a bi-product of this proof, and the equivalence of conditions (1) and (2) in Theorem~\ref{main1} is an immediate consequence of Theorem~\ref{foliation}.
We begin with the following proposition which is the key observation.
(Given an analytic set $Z\subset\Omega$, we denote the set of regular points of $Z$ by $Z_{{\rm reg}}$, the set of singular points by $Z_{{\rm sing}}$, and the inclusion map of $Z_{{\rm reg}}$ into $\Omega$ by $i_{Z_{{\rm reg}}}$.)

\begin{proposition}\label{main}
Let $\Omega\subset\mathbb C^n$ be a bounded domain, and let $h_j\in PH(\Omega)\cap\cC(\overline\Omega)$
for $\jN$.  Suppose there is an irreducible analytic set $Z\subset\Omega$ of dimension $d\geq 1$ such that $i^*_{Z_{{\rm reg}}}(\overline\partial h_{i_1}\wedge\cdot\cdot\cdot\wedge\overline\partial h_{i_d})\equiv 0$
for all $(i_1,\ldots,i_d)$.  Then $\mathcal G_h(b\Omega)$ is not polynomially convex. 
\end{proposition}

\begin{lemma}\label{analyticset}
Let $\Omega\subset\mathbb C^n$ be a domain, let $h_j\in PH(\Omega)$
for $\jN$, and let $Z\subset\Omega$ be an irreducible analytic set of dimension $d\geq 1$.
Let $1\leq m\leq d$, fix $(i_1,\ldots,i_m)$, and define 
$$
Z':=\{z\in Z_{{\rm reg}}:i^*_{Z_{{\rm reg}}}(\overline\partial h_{i_1}\wedge\cdot\cdot\cdot\wedge\overline\partial h_{i_m})(z)=0\}.
$$
Then $\tilde Z:=Z'\cup Z_{{\rm sing}}$ is an analytic subset of $\Omega$.
\end{lemma}

\begin{proof}
Let $z_0\in \Omega$, and let $U_0$ be a simply connected neighborhood of $z_0$. Then for each $\jN$ there is a 
$g_j\in \mathcal{O}(U_0)$ such that $h_j+g_j$ is real and then also $f_j\in \mathcal{O}(U_0)$ such that
$Re(f_j)=h_j+g_j$. Notice that
\begin{equation}\label{torsk}
\debar h_j = \debar Re(f_j) = \overline{df_j}/2,
\end{equation} 
so that
\begin{equation}\label{istar}
i^*_{Z_{{\rm reg}}}(\debar h_{i_1}\wedge \cdots \wedge \debar h_{i_m})=0 \hbox{\ if and only if\ }
i^*_{Z_{{\rm reg}}}(df_{i_1}\wedge \cdots \wedge df_{i_m})=0.
\end{equation}
Hence, $Z'$ is an analytic subset
of $Z_{{\rm reg}}$.

Assume now that $z_0\in Z_{{\rm sing}}$ and choose (possibly after shrinking $U_0$) generators 
$\psi_1,\ldots,\psi_{\ell}$ for the radical ideal sheaf in $U_0$ of holomorphic functions
vanishing on $Z$.  Let
\begin{equation*}
W=\{\, z\in U_0: df_{i_1}\wedge \cdots \wedge df_{i_m}\wedge d\psi_{j_1}\wedge \cdots \wedge d\psi_{j_{n-d}}(z)=0\quad \forall \, (j_1,\ldots,j_{n-d})\}.
\end{equation*}
Then $W$ is an analytic subset of $U_0$.  A point $z\in Z\cap U_0$ lies in $Z_{{\rm reg}}$ if and only there is some choice of $(j_1,\ldots,j_{n-d})$ such that 
$d\psi_{j_1}\wedge \cdots \wedge d\psi_{j_{n-d}} (z) \neq 0$.  It follows that $W$ contains $Z_{{\rm sing}}\cap U_0$.
It also follows that for $z\in Z_{{\rm reg}}\cap U_0$ we have
$i^*_{Z_{{\rm reg}}}(df_{i_1}\wedge \cdots \wedge df_{i_m}) (z)=0$ if and only if 
$(df_{i_1}\wedge \cdots \wedge df_{i_m}\wedge d\psi_{j_1}\wedge \cdots \wedge d\psi_{j_{n-d}})(z)=0$ for all $(j_1,\ldots,j_{n-d})$.
Consequently, (\ref{istar}) gives $W\cap Z_{{\rm reg}}=Z'\cap U_0$, and the lemma follows. 
\end{proof}

\begin{proof}[Proof of Proposition \ref{main}]
Let $m$ be the largest integer such that there exist $h_{i_1},\ldots,h_{i_m}$ such that 
$i^*_{Z_{{\rm reg}}}(\debar h_{i_1}\wedge \cdots \wedge \debar h_{i_m})$ does not vanish identically; by hypothesis,
$m<d$. If all $h_i$ are holomorphic on $Z$ then the hull of $\mathcal{G}_h(b\Omega)$ contains 
$\mathcal{G}_h(Z)$ and we are done; we may thus assume that $m\geq 1$. 

We will first define a singular foliation $\mathcal{F}$ of $Z$ by non-trivial varieties such that all the $h_i$ are holomorphic along 
the plaques; see Definition~\ref{singularfol} for the precise meaning of this.
As in the beginning of the proof of Lemma~\ref{analyticset}
every $z_0\in \Omega$ has a neighborhood $U_0$ such that for all $j$, $Re(f_j)=h_j+g_j$, where the $f_j$ and the $g_j$ are
holomorphic in $U_0$. Notice that the $f_j$ (and the $g_j$) are uniquely defined up to adding constants;
the level sets of the map $F=(f_1,\ldots,f_N)\colon U_0 \to \C^N$ thus unambiguously defines a partitioning of $U_0$
by varieties. We define $\mathcal{F}$ by intersecting with $Z$. Clearly all the $h_j$ are holomorphic along each plaque.
We need to show that each plaque has dimension $\geq 1$. Let 
\begin{equation*}
\tilde{Z}=Z_{{\rm sing}}\cup \bigcap_{(i_1,\ldots,i_m)} \{z\in Z_{{\rm reg}}: i^*_{Z_{{\rm reg}}}(\debar h_{i_1}\wedge \cdots \wedge \debar h_{i_m})=0\},
\end{equation*}  
which is a proper analytic subset of $Z$ by Lemma~\ref{analyticset}.
Let $z_0\in Z\setminus \tilde{Z}\subset Z_{{\rm reg}}$ and assume that 
$i^*_{Z_{{\rm reg}}}(\debar h_{1}\wedge \cdots \wedge \debar h_{m})(z_0)\neq 0$. 
Then by \eqref{torsk}, 
$i^*_{Z_{{\rm reg}}}(df_{1}\wedge \cdots \wedge df_{m})(z_0)\neq 0$ and so we can choose local coordinates $\zeta_1,\ldots,\zeta_d$
for $Z_{{\rm reg}}$ centered at $z_0$ such that $f_1-f_1(z_0)=\zeta_1,\ldots,f_m-f_m(z_0)=\zeta_m$. 
By \eqref{torsk} and the choice of $m$ we have 
$i^*_{Z_{{\rm reg}}}(df_1\wedge \cdots \wedge df_m \wedge df_j)\equiv 0$ for all $j$, and so 
$\partial f_j/\partial \zeta_k \equiv 0$
for $j=m+1,\ldots,N$ and $k=m,\ldots,d$. Hence, the level set of $f_j|_Z$ through $z_0$ contains the common level set of 
$f_1|_Z,\ldots,f_m|_Z$ through $z_0$; the plaque through $z_0$ thus equals the latter set which has dimension $d-m\geq 1$.
Now, the function $Z\ni z \mapsto \textrm{dim}_{z}\, Z\cap (F^{-1}(F(z)))$ is $\geq 1$ on $Z\setminus \tilde{Z}$ and upper semicontinuous, see, \emph{e.g.},  
\cite[Ch.\ 2, Prop.\ 8.2]{Demailly}. 
Hence, the plaque through any point of $Z$ has dimension $\geq 1$.

\smallskip

Assume now to get a contradiction that $\mathcal{G}_h(b\Omega)$ is polynomially convex. Then there is a point $z_0\in Z$ and a polynomial
$P(z,w)$ in $\C^{n+N}$ such that
\begin{equation*}
\sup_{z\in \overline{Z}} |P(z,h(z))| = |P(z_0,h(z_0))| > \sup_{z\in b\Omega} |P(z,h(z))|.
\end{equation*}
It follows that the set $K:=\{z\in Z: P(z,h(z))=P(z_0,h(z_0))\}$ 
is a compact subset of $\Omega$. Since the $h_j$ are holomorphic
along the plaques of $\mathcal{F}$ it follows from the maximum principle (see, \emph{e.g.}, \cite[Ch.\ 4, Thm.\ 2G]{Whitney}) 
applied to $z\mapsto P(z,h(z))$ that
$K$ is a union of plaques of $\mathcal{F}$. Let $(p_1,\ldots,p_n)\in K$ be a point where $|z_1|$ attains its maximum on $K$. 
Then $K_1:=K\cap \{z_1=p_1\}$ is compact and non-empty and is, again by the maximum principle, a union of plaques of $\mathcal{F}$. 
Repeating for the rest of the coordinate functions $z_2,\ldots,z_n$ we see that there is a plaque of $\mathcal{F}$ contained in a point,
a contradiction.
\end{proof}

\begin{proof}[Proof of Theorem \ref{main1}] 
$(1) \Rightarrow (3)$:  
Assume that $\mathcal{G}_h(b\Omega)$ is polynomially convex and 
let $\varphi\colon\triangle\hookrightarrow\Omega$ be a holomorphic embedding.  
We will show that there exist $h_{i_1},\ldots,h_{i_d}$ and an irreducible analytic set $Z\subset\Omega$ of dimension $d$ 
such that $\dim(\varphi(\triangle)\cap(Z\setminus \tilde{Z}))=1$, where
$$
\tilde{Z}:=\{z\in Z_{{\rm reg}}:i^*_{Z_{{\rm reg}}}(\overline\partial h_{i_1}\wedge\cdots\wedge\overline\partial h_{i_d})(z)=0\}\cup Z_{{\rm sing}}.
$$
In that case, clearly all of the $h_{i_j}$ cannot be holomorphic along $\varphi(\triangle)$.

To obtain the analytic set $Z$, first we let 
\begin{equation}\label{Z1}
Z_1=\{z\in\Omega:\overline\partial h_{i_1}\wedge\cdots\wedge\overline\partial h_{i_n}(z)=0 \quad \forall (i_1,\ldots, i_n)\}.
\end{equation}
Then $\dim(Z_1)<n$ by Proposition~\ref{main} since $\mathcal{G}_h(b\Omega)$ is polynomially convex.   
If $\varphi(\triangle)$ is not contained in $Z_1$ then $\Omega$
works as $Z$ and we are done.  
Otherwise $\varphi(\triangle)$ is contained in an irreducible component of $Z_1$; abusing notation we denote this component 
by $Z_1$ and we define 
$$
Z_2=\{z\in ({Z_1})_{{\rm reg}}:i^*_{Z_1} (\overline\partial h_{i_1}\wedge\cdots\wedge\overline\partial h_{i_{d_1}})(z)=0 
\quad \forall (i_1,\ldots,i_{d_1})\}\cup (Z_1)_{{\rm sing}}, 
$$
where $d_1=\dim(Z_1)$.
By Lemma~\ref{analyticset} we have that $Z_2$ is an analytic subset of $\Omega$ and by Proposition \ref{main}
we have that $\dim(Z_2)<\dim(Z_1)$.  If $\varphi(\triangle)$ is not contained in $Z_2$ then $Z_1$ works as $Z$ and we are done.   
Repeating this process we eventually find the desired analytic set $Z$.

$(3)\Rightarrow (4)$: The proof of \cite[Theorem~1.3]{SW} given in \cite{SW} in fact shows that this implication holds.

$(4)\Rightarrow (1)$: By \cite[Lemma~4.6]{SW} $\mathcal G_h(\oO)$ is polynomially convex.  
Thus $\widehat{\mathcal G_h(b\Omega)}$ is contained in
$\mathcal G_h(\oO)$.  For any point $a$ in $\Omega$, by (3) there is a function $g$ in $[z,h]_{\oO}$ 
such that $g(p)=1$ and $g|b\Omega=0$.  Consequently there is a polynomial $p$ in $\C^{n+N}$ such that 
$p\bigl(a,h(a)\bigr)>\sup_{z\in b\Omega}p\bigl(z,h(z)\bigr)$, so $\bigl(a,h(a)\bigr)$ is not in $\widehat{\mathcal G_h(b\Omega)}$.  
Thus ${\mathcal G_h(b\Omega)}$ is polynomially convex.

$(1) \Rightarrow (2)$: Obvious.

$(2) \Rightarrow (1)$: This is immediate from Theorem~\ref{foliation} which we prove next. 
\end{proof}

\begin{proof}[Proof of Theorem~\ref{foliation}]
Assume that $\mathcal{G}_h(b\Omega)$ is not polynomially convex. Then by Theorem~\ref{main1} (2) there is an
embedded analytic disk $\varphi\colon \triangle \to \Omega$ such that each $h_j$ is holomorphic along $\varphi(\triangle)$.
The procedure in the proof of Theorem~\ref{main1}, $(1)\Rightarrow (2)$, gives a decreasing sequence $Z_1\supset Z_2 \supset \cdots$
of irreducible analytic subsets of $\Omega$ such that $\varphi(\triangle)\subset Z_j$ for all $j$.  
For dimensional reasons it follows that for some $k$, $Z_k=Z_{k+1}=\cdots$ and 
$\dim\,Z_k\geq 1$.
Moreover, from the construction of the sets $Z_j$ we have that 
\begin{equation*}
i^*_{Z_k} \left(\overline\partial h_{i_1} \wedge \cdots \wedge \overline\partial h_{i_{d_k}}\right) \equiv 0,
\quad \forall (i_1,\ldots,i_{d_k}),
\end{equation*}  
where $d_k=\dim\, Z_k$. As in the first part of the proof of Proposition~\ref{main}
we then get a singular foliation of $Z_k$ by non-trivial varieties such that all the $h_j$ are strongly holomorphic 
along the plaques.  
That $\mathcal G_h(Z)$ is contained in $\widehat{\mathcal G_h(b\Omega)}\setminus \mathcal G_h(b\Omega)$ is proved by repeating the argument in the last paragraph of the proof of Proposition~\ref{main}
\end{proof}


\section{Algebras generated by smooth functions}\label{noanalyticstructure}

In this section we prove Theorems~\ref{graphtorus}--\ref{smoothexample2}.
We will use the following version of \cite[Proposition~4.7]{SW}.

\begin{proposition}\label{proppen}
Let $K\subset \C^n$ be a compact set, let $F\colon K\to \C^k$ be the restriction to $K$ of a polynomial map, and
assume that $[w_1,\ldots,w_k]_{F(K)}=\mathcal{C}(F(K))$. Then the following holds:
\begin{itemize}
\item[(i)] $\widehat{K} = \bigcup_{w\in \C^k} \widehat{F^{-1}(w)}$

\item[(ii)] If $[z_1,\ldots,z_n]_{F^{-1}(w)} = \mathcal{C}(F^{-1}(w))$ for all $w\in \C^k$,
then $[z_1,\ldots,z_n]_{K} = \mathcal{C}(K)$.
\end{itemize}
\end{proposition}

Recall that we denote by $\Gamma$ the distinguished boundary of the bidisk.  

\begin{lemma}\label{hullcor}
Let $K\subset \Gamma$ be a closed subset such that for some $c\in S^1$ the set 
$K$ is disjoint from the circle $\{\, (z_1,z_2)\in \Gamma: z_1=c\,\}$, and there is no $a\in S^1$ such that $K$ contains the full circle $\{\, (z_1,z_2)\in \Gamma: z_1=a\,\}$.  Then $[z_1,z_2]_K=\mathcal{C}(K)$, and in particular, $K$ is polynomially  convex.
\end{lemma}

\begin{proof}
First note that if $L$ is a proper closed subset of $S^1\subset \C$, then $[z]_L=\mathcal{C}(L)$.  (See for instance \cite[pp.~81--82]{Hoffman}.)
Let $F$ be the restriction to $K$ of the projection $(z_1,z_2)\mapsto z_1$. The assumptions on $K$
imply that $F(K)$ is a proper closed subset of $S^1$ and that $F^{-1}(a)$ is a proper closed subset
of $\{a\}\times S^1$ for each $a\in \C$. The hypotheses of Proposition~\ref{proppen} (ii) are thus satisfied and the result follows. 
\end{proof}

\begin{lemma}\label{function}
If $K$ is a closed subset of a smooth manifold $M$ such that $M \setminus K$ is disconnected and $U$ and $V$ form a separation of $M \setminus K$, then there is a  smooth function $f$ on $M$ that is identically zero on $K$, strictly positive everywhere on $U$, and strictly negative everywhere on $V$.
\end{lemma}

\begin{proof}
Every closed subset of a smooth manifold is the zero set of some smooth real-valued function on the manifold (see \cite[pp.\ 76--77]{Spivak}),
so let $g$ be a smooth real-valued function on $M$ with $K$ its zero set. By replacing $g$ by $g^2$, 
we may assume that $g\geq 0$. Let $\gamma \colon \R \to \R$ be a smooth function that is identically zero for $x\leq 0$
and strictly positive for $x >0$. Then the function $\gamma \circ g:M\rightarrow \R$ is smooth and has $K$ as its zero set, and it is easily verified that
all partial derivatives of  
$\gamma\circ g$ vanish identically on $K$. It follows that the function $f$ defined by setting $f$ 
equal to $\gamma\circ g$ on $K\cup U$ and equal to $-\gamma\circ g$ on $V$ has the required properties.
\end{proof}

The key to our proof of Theorem~\ref{graphtorus} is the result of Herbert Alexander 
\cite{Alexander} that asserts the existence of a closed set $E\subset \Gamma$ such that 
$\hat{E}\setminus E$ is nonempty but contains no analytic disk.
However, we need to show that the set can be taken to have a certain additional property.  This is achieved in the following lemma.

\begin{lemma}\label{lma1}
There is a closed subset $X$ of the torus $\Gamma$ such that 
\item{\rm(i)} $\hat{X}\setminus X$ is nontrivial but contains no
analytic disk, and 
\item{\rm(ii)} $\Gamma\setminus X$ is the disjoint union of open sets $U$ and $V$ each
of which is disjoint from some circle $\{(z_1,z_2)\in \Gamma: z_1=c\}$ (depending on U and V).
\end{lemma}

We postpone the proof of this lemma and first use it to prove Theorem~\ref{graphtorus}.

\begin{proof}[Proof of Theorem~\ref{graphtorus}]

Consider the graph $K:=\mathcal{G}_f(\Gamma)\subset \C^3$ and let
$F$ be the restriction to $K$ of the projection
$(z_1,z_2,z_3) \mapsto z_3$. 
Then $F(K)$ is a compact subset of $\R$ so, by Weierstrass' approximation
theorem, the hypotheses of Proposition~\ref{proppen} (i) are fulfilled and we can determine the hull of  
$K$ by determining the hulls of the fibers $F^{-1}(r)=\{f=r\}\times \{r\}$. 
To do this, first notice that $X$ has to intersect every circle
$\{(z_1,z_2)\in \Gamma : z_1=c\}$. Indeed, $X$ cannot contain such a full circle since $\hat{X}\setminus X$ contains no analytic disk
and so if $X$ were disjoint from some such circle, then Corollary~\ref{hullcor} would imply that $X$ were polynomially convex.
It follows that $\{f=r\}$, $r\neq 0$, contains no full circle $\{(z_1,z_2)\in \Gamma : z_1=c\}$. Furthermore,
since $\{f=r\}$ is contained in either $U$ or $V$ for $r\neq 0$, $\{f=r\}$ is disjoint from some such circle.
Hence, for $r\neq 0$, the fiber $F^{-1}(r)=\{f=r\}\times \{r\}$ is polynomially convex by Corollary~\ref{hullcor}.
The hull of $F^{-1}(0) = X\times \{0\}$ is $\hat{X}\times \{0\}$ and from Proposition~\ref{proppen} we get that
\begin{equation*}
\widehat{\mathcal{G}_f(\Gamma)} \setminus \mathcal{G}_f(\Gamma) =
(\hat{X}\setminus X) \times \{0\},
\end{equation*}
which is nonempty but contains no analytic disk.
\end{proof}

The proof of Lemma~\ref{lma1} uses the following lemma.

\begin{lemma}\label{lma2}
If $K \subset \Gamma$ and $C\subset \{(z_1,z_2)\in \Gamma:\, z_1=c\}$ are closed sets such that 
$K\cup C$ does not contain the full circle $\{(z_1,z_2)\in \Gamma:\, z_1=c\}$, then  
$(K\cup C)^{\widehat{}}=\widehat{K}\cup C$. The same holds with $z_1$ replaced by $z_2$.
\end{lemma}

\begin{proof}
Clearly $\overline{\mathbb{D}}^2\supset (K\cup C)^{\widehat{}} \supset \widehat{K}\cup C$ so we are to show that 
if $a=(a_1,a_2) \in \overline{\mathbb{D}}^2 \setminus (\widehat{K}\cup C)$, then $a$ is not in $(K\cup C)^{\widehat{}}$. 
Assume first that $a_1\neq c$ and let $q(z)=z_1-c$; then $q=0$ on $C$ and $q(a)\neq 0$. 
Since $a$ is not in $\hat{K}$, there is a polynomial $p$ such that $p(a)=1$ and $\sup_{z\in K} |p(z)|<1$.
Some power $p^n$ of $p$ satisfies $\sup_{z\in K} |p^n(z)|< |q(a)|/(1+\sup_{z\in K}|q(z)|)$ and it follows that
$q\cdot p^n$ separates $a$ from $K\cup C$.

Assume now that $a_1=c$. Since each point of $\Gamma$ is a peak point for the bidisk algebra, we may assume that 
$|a_2|<1$. Choose $\theta$ such that $e^{i\theta}=c$ and choose $\alpha_1,\beta_1,\alpha_2,\beta_2$ with 
$\alpha_1<\theta<\beta_1$ such that the ``rectangle'' $\{(e^{is}, e^{it}):\, s\in [\alpha_1,\beta_1], t\in [\alpha_2,\beta_2]\}$
is disjoint from $K\cup C$; this is possible since $K\cup C$ is closed and does not contain the full circle
$\{(z_1,z_2)\in \Gamma:\, z_1=c\}$.
Since the set $J=S^1\setminus \{e^{it}:\, t\in [\alpha_2,\beta_2]\}$ is polynomially convex, there is a 
polynomial $q$ of one variable such that $q(a_2)=1$ and $\sup_{w\in J} |q(w)| <1$. Let $M$ be the supremum of 
$|q|$ over the unit disk. For sufficiently large $n$, the supremum of $|(1+\bar{c}z)/2|^n$ over 
$S^1\setminus \{e^{is}:\, s\in [\alpha_1,\beta_1]\}$ is strictly less than $1/M$ and so
\begin{equation*}
\sup_{(z_1,z_2)\in K\cup C} \left|\left(\frac{1+\bar{c}z_1}{2}\right)^n q(z_2) \right| < 1 =
\left(\frac{1+\bar{c}z_1}{2}\right)^n q(z_2)\big|_{(z_1,z_2)=(a_1,a_2)}.
\end{equation*} 

Obviously the roles of $z_1$ and $z_2$ can be reversed.
\end{proof}

\begin{proof}[Proof of Lemma~\ref{lma1}]
By \cite{Alexander} there exists a closed set $E\subset \Gamma$ such that 
$\hat{E}\setminus E$ is nonempty but contains no analytic disk. We will obtain
the set $X$ by taking the union of $E$ with certain circle segments.

Since $E$ is a proper closed subset of $\Gamma$ there exist $\alpha_1, \beta_1, \alpha_2,\beta_2$
with $-\pi < \alpha_1 < \beta_1 < \pi$ and $-\pi < \alpha_2 < \beta_2 < \pi$
such that the ``rectangle'' $\{(e^{is}, e^{it}): s\in [\alpha_1,\beta_1], t\in [\alpha_2,\beta_2]\}$ is disjoint from $E$.
Choose $\varphi_1,\psi_1, \varphi_2, \psi_2$ and $\varphi'_1,\psi'_1, \varphi'_2, \psi'_2$ such
that $\alpha_1 < \varphi_1<\psi_1<\varphi'_1<\psi'_1<\beta_1$ and
$\alpha_2 < \varphi_2=\varphi'_2<\psi_2=\psi'_2<\beta_2$. Let 
\begin{eqnarray*}
\Sigma &=& \big\{(e^{i\varphi_1}, e^{it}):\, t\in [-\pi,\pi]\setminus [\varphi_2,\psi_2] \big\}\cup 
\big\{(e^{i\psi_1}, e^{it}):\, t\in [\varphi_2,\psi_2] \big\} \\
& & \cup \big\{(e^{is}, e^{i\varphi_2}):\, s\in [\varphi_1,\psi_1] \big\} \cup 
\big\{(e^{is}, e^{i\psi_2}):\, s\in [\varphi_1,\psi_1] \big\},
\end{eqnarray*}
let $\Sigma'$ be given by the same expression but with primes throughout, and let $X=E\cup \Sigma\cup\Sigma'$.
It is easily seen that $\Gamma\setminus X$ is disconnected and can be written as the disjoint union of open sets 
$U$ and $V$ each of which is disjoint from some circle $\{(z_1,z_2)\in \Gamma: z_1=c\}$. 
By repeated use of Lemma~\ref{lma2} it follows that $\hat{X}\setminus X = \hat{E}\setminus E$ and so
$\hat{X}\setminus X$ is nonempty but contains no analytic disk. 
\end{proof}

\bigskip

The proof of Theorem~\ref{graphball} uses the following lemma.

\begin{lemma}\label{nonconstant}
Let $X\subset b\B_2$ be a compact set that is disjoint from the circle $\{\, (z_1,z_2)\in b\B_2: |z_1|=1,\  z_2=0\,\}$ and contains no full circle $\{\, (z_1,z_2)\in b\B_2: z_1=c\,\}$ for $c$ a constant with $|c|<1$.   Then there is a smooth real-valued function $g$ on $b\B_2$ with $X$ as its zero set and that is nonconstant on each circle $\{\, (z_1,z_2)\in b\B_2: z_1=c\,\}$ for $|c|<1$.
\end{lemma}

We postpone the proof of this lemma and first use it to prove Theorem~\ref{graphball}.

\begin{proof}[Proof of Theorem \ref{graphball}] 
By the theorem of Duval and Levenberg~\cite{DL} quoted in the introduction as Theorem~\ref{duval} there is a compact set $X\subset b\B_2$ that is not polynomially convex but whose polynomially convex hull contains no analytic disk.  This example is also presented in \cite{Stout}.  Examination of the proof as presented in \cite{Stout} reveals that one can easily arrange to have $X$ be disjoint from the circle $\{\, (z_1,z_2)\in b\B_2: |z_1|=1,\  z_2=0\,\}$.  (In Theorem~1.7.2 take the set $K$ to be disjoint from the set $\{\, z_2=0\,\}$ and in the proof take $z_2$ to be among the polynomials $Q_j$.)

Let $g$ be the function given by Lemma~\ref{nonconstant}, and set $f_1=g$ and $f_2=x_1\cdot g$.  Consider the graph $K:=\mathcal{G}_f(b\B_2)\subset \C^4$ and let
$F$ be the restriction to $K$ of the projection
$(z_1,\ldots,z_4) \mapsto (z_3, z_4)$. 
Then $F(K)$ is a compact subset of $\R^2$ so, by Weierstrass' approximation
theorem, the hypotheses of Proposition~\ref{proppen} (i) are fulfilled and we can determine the hull of  
$K$ by determining the hulls of the fibers $F^{-1}(r)$.  Since 
the hull of $F^{-1}(0,0) = X\times \{(0,0)\}$ is $\hat{X}\times \{(0,0)\}$, and the set $\hat{X}\setminus X$ is nonempty but contains no analytic disk, it suffices to show that each of the other fibers is polynomially convex.

Let $E=F^{-1}(r_1,r_2)$ be some other fiber.   Note that on $E$ we have $z_3=r_1\neq 0$ and $x_1=z_4/z_3=r_2/r_1$.  In particular $E$ is contained in a level set of $x_1$.  
Let $G$ be the restriction to $E$ of the mapping $(z_1,\ldots,z_4) \mapsto (z_3, (z_1-(r_2/r_1))/i)$.  We will establish polynomial convexity of $E$ by applying Proposition~\ref{proppen} (i) again, this time $K$ replaced by $E$ and $F$ replaced by $G$.
Since $(z_1-(r_2/r_1))/i=y_1$ on $E$, the set $G(E)$ is contained in $\R^2$ so, by Weierstrass' approximation
theorem again, the hypotheses of Proposition~\ref{proppen} (i) are fulfilled.  
On each fiber $G^{-1}(w_1,w_2)$, the functions $z_1$, $z_3$, and $z_4$ are constant.  Because $g$ is nonconstant on each circle $\{\, (z_1,z_2)\in b\B_2: z_1=c\,\}$ for $|c|<1$, it follows that each fiber is either a {\it proper} subset of some circle where $z_1$, $z_3$, and $z_4$ are constant or else is a single point.  Thus each fiber  is polynomially convex and hence so is $E$.

\end{proof}

The proof of Lemma~\ref{nonconstant} uses the following calculus lemma for which the reader can easily supply a proof.

\begin{lemma}\label{calculus}
Suppose that $f$ is a smooth function on $\R^n\setminus\{0\}$ and to each partial derivative $\partial^{k_1+\cdots +k_n} \, f/\partial x_1^{k_1}\cdots \partial x_n^{k_n}$ of any order there corresponds an integer $k$ such that $\partial^{k_1+\cdots +k_n} f/\partial x_1^{k_1}\cdots \partial x_n^{k_n}$ blows up at the origin no faster than $1/r^k$.  Suppose also that $\alpha$ is a smooth function on $\R^n$ with all partial derivatives of all orders equal to zero at the origin.  Then the function $\alpha\cdot f$ is smooth on $\R^n$.  (Of course here we define $(\alpha\cdot f)(0)$ to be zero.)
\end{lemma}

\begin{proof}[Proof of Lemma \ref{nonconstant}] 

Let $\rho$ be a smooth function on $b\mathbb{B}^2$ such that $X=\{\rho=0\}$ (which recall is possible by \cite[pp.~76--77]{Spivak}).  By replacing
$\rho$ by $\rho^2$ and rescaling, we may assume that $0\leq \rho \leq 1$. Let $\pi \colon b\mathbb{B}^2 \to \C$
be the restriction to $b\mathbb{B}^2$ of the projection $\pi(z_1,z_2)=z_1$ and let $X'=\pi(X)$. By 
the hypotheses on $X$, note that $X'$ is a compact subset of the open unit disk $\Delta \subset \C$.  For each $c\in \Delta$, denote the circle $b\mathbb{B}^2\cap \{z_1=c\}$ by $S_c$.  For functions defined on $b\mathbb{B}^2$, let \lq\lq subscript $c$" denote restriction to $S_c$, so for instance $\rho_c$ denotes the restriction of $\rho$ to $S_c$.  Notice that if $c\in X'$, then $\rho$ is non-constant on the circle $S_c$ since 
$X$ does not contain $S_c$. By compactness of $X'$ and continuity of $\rho$ there is a $\delta >0$ and 
a neighborhood $U\Subset \Delta$ of $X'$ such that 
\begin{equation}\label{eq1}
\textrm{max}\, \rho_c - \textrm{min}\, \rho_c > \delta, \quad \forall c\in U.
\end{equation}

Let $\chi_1=\chi_1(z_1)$ be a smooth function such that $\chi_1 =1$ in a neighborhood of $X'$, the support of $\chi_1$ is contained in $U$, and $0\leq \chi_1 \leq 1$; let $\chi_2=1-\chi_1$. We will consider $\chi_j$ as a function on $\C^2$ or $b\mathbb{B}^2$
that is independent of $z_2$. For $k\in \N$ let
\begin{equation*}
\psi_k(z_2)= 1+ e^{-1/|z_2|} \sin(k \cdot \textrm{arg}(z_2)).
\end{equation*} 
Notice that $\psi_k$ is smooth, by Lemma~\ref{calculus}, and that $0<\psi_k<2$. We will show, that for sufficiently large $k$, the function $g\colon b\mathbb{B}^2 \to \R$ defined by
\begin{equation*}
g:= \rho \chi_1 + \psi_k \chi_2
\end{equation*}
has the desired properties.

Clearly the zero set of $g$ equals $X$.  Given $c\in \Delta$ we must show that 
$g_c$ is non-constant. For notational convenience we write $z_2=re^{i\theta}$;
on the circle $S_c$ thus $z_2=\sqrt{1-|c|^2}\,e^{i\theta}$.
If $c\in \Delta \setminus U$ then
$g_c= 1+e^{-1/\sqrt{1-|c|^2}} \sin(k \cdot \theta)$, which is non-constant. 
Assume that $c\in U$ is such that $\chi_2(c)\geq \delta/3$. Differentiating $g_c$ with respect to
$\theta$ we get
\begin{equation}\label{eq2}
\frac{d\rho_c}{d\theta} \chi_1(c) + k e^{-1/\sqrt{1-|c|^2}} \cos (k\theta) \chi_2(c).
\end{equation}
Since $c\in U\Subset \Delta$, $\chi_2(c)\geq \delta/3$, and $d\rho_c/d\theta$ is uniformly bounded for $c\in U$, it follows that, taking $\theta=0$,
the second term in \eqref{eq2} dominates the first one for $k$ sufficiently large. Hence, $dg_c/d\theta$ cannot be identically
zero and so $g_c$ is non-constant.
Finally, assume that $c\in \{\chi_2<\delta/3\}\subset U$. Then
\begin{equation}\label{eq3}
\textrm{max}\, g_c > \textrm{max}\, \rho_c \cdot \chi_1(c) > 
\textrm{max}\, \rho_c \cdot (1-\delta/3) \geq \textrm{max}\, \rho_c - \delta/3,
\end{equation}
\begin{equation}\label{eq4}
\textrm{min}\, g_c < \textrm{min}\, \rho_c \cdot \chi_1(c)+ 2\chi_2(c) < \textrm{min}\, \rho_c + 2\delta/3.
\end{equation}
From \eqref{eq1}, \eqref{eq3}, and \eqref{eq4} it thus follow that
\begin{equation*}
\textrm{max}\, g_c - \textrm{min}\, g_c > \textrm{max}\, \rho_c - \textrm{min}\, \rho_c - \delta =0
\end{equation*}
and so $g_c$ is non-constant.

\end{proof}

The proofs of Theorems~\ref{smoothexample1} and ~\ref{smoothexample2} use the following easy lemma.
The same principle has been used by the first author to construct other counterexamples to the peak point conjecture \cite{Izzo}.

\begin{lemma}\label{generators}
Suppose $X$ is a compact Hausdorff space, $E\subset X$ is closed, $B$ is a uniform algebra on $E$, and $A=\{\, f\in \cC(X): f|E\in B\,\}$.  If $\{f_\alpha\}$ is a collection of functions in $A$ such that $\{f_\alpha|E\}$ generates $B$, and $\{g_\beta\}$ is a collection of real-valued functions that generates $\cC(X)$, and $\rho\in \cC(X)$ is a real-valued function that vanishes precisely on $E$, then the collection $\{f_\alpha\}\cup \{\rho\} \cup \{\rho g_\beta\}$ generates the algebra $A$.
\end{lemma}

\begin{proof}
One trivially verifies that the collection of functions $\{\rho\} \cup \{\rho g_\beta\}$ has $E$ as its common zero set and separates points off of $E$, so these functions induce real-valued functions that separate points on the quotient space $X/E$ obtained from $X$ by identifying $E$ to a point.  Consequently, the Stone-Weierstrass theorem shows that the induced functions generated $\cC(X/E)$.  Hence the collection $\{\rho\} \cup \{\rho g_\beta\}$ generates the algebra $\{\, f\in \cC(X): f\ \hbox{is constant on\ } E\}$.  It follows that the collection $\{f_\alpha\}\cup \{\rho\} \cup \{\rho g_\beta\}$ generates the algebra $A$.
\end{proof}

\begin{proof}[Proof of Theorem~\ref{smoothexample1}]
This is essentially proved in \cite{Izzo}.  Let $K$ be a compact planar set that has no interior and is contained in $\Omega$ such that $\cR(K)\neq \cC(K)$.  (See, for instance, \cite[p.~25--26]{Gamelin} for the existence of such a set.)    By \cite[Lemma~11]{Basener} there is a $C^\infty$ function $g$ such that the functions $z$ and $g$ generate $\cR(K)$.   Let $\rho$ be a $C^\infty$ function on $\oO$ that vanishes precisely on $K$.
Let $\cA=\{ \, f\in \cC(\oO): f|K\in \cR(K)\, \}$.  By Lemma~\ref{generators}, $[z,g,\rho, \rho\overline z]=[z,g,\rho, \rho x,\rho y]=\cA$.   The maximal ideal space of $\cA$ is $\oO$ by \cite[Theorem~4]{Bear}, and hence $\mathcal G_{(g,\rho,\overline z\rho)}(\oO)$ is polynomially convex.  Taking $\{f_1,f_2,f_3\}=\{g,\rho, \overline z\rho\}$ yields the theorem.
\end{proof}

\begin{proof}[Proof of Theorem~\ref{smoothexample2}]
The proof is similar to the previous one.
By \cite[Theorem~4]{Basener} (or see \cite[Example~19.8]{Stout}) there is a compact set $K\subset \C^2\subset\C^n$ such that $K$ is rationally convex and every point of $K$ is a peak point for $\cR(K)$ but $\cR(K)\neq \cC(K)$.  By translating and rescaling, we may assume that $K$ is contained in $\Omega$.   By \cite[Lemma~11]{Basener} there is a $C^\infty$ function $g$ such that the functions $\row zn, g$ generate $\cR(K)$.  Let $\rho$ be a $C^\infty$ function on $\oO$ that vanishes precisely on $K$.
Let $\cA=\{ \, f\in \cC(\oO): f|K\in \cR(K)\, \}$.  By Lemma~\ref{generators}, 
$[\row zn, g, \rho, \rho\overline{z}_1,\ldots, \rho\overline{z}_n]=[\row zn, g, \rho, \row {\rho x}n, \row {\rho y}n]=\cA$.  The maximal ideal space of $\cA$ is $\oO$ by \cite[Theorem~4]{Bear}.  Now taking $\{\row fN\}=\{\row zn, g, \rho, \rho\overline{z}_1,\ldots, \rho\overline{z}_n\}$ yields the theorem.
\end{proof}

\begin{proof}[Proof of Theorem~\ref{embedM}]
We divide the proof into steps.

{\it Step 1\/}: We show that for some $N$ there is an embedding $\Phi$ of $M$ into $\C^2\times \R^N \subset \C^{2+N}$ 
such that $\Phi(M)\cap (\C^2 \times \{0\}^N) = \Gamma\times \{0\}^N$ and $\Phi(M)\cap (\C^2 \times \{r\})$ 
contains at most one  point for each $r\in \R^N$ with $r\neq 0$; recall that $\Gamma$ is the distinguished
boundary of the bidisk in $\C^2$.

First choose a 2-torus $T$ that is a smoothly embedded submanifold of $M$, and choose complex-valued functions 
$g_1, g_2$ such that $(g_1,g_2):T\rightarrow \C^2$ gives a diffeomorphism of $T$ onto $\Gamma$.  
Extend $g_1,g_2$ to smooth functions defined on all of $M$ which we continue to denote by 
$g_1,g_2$.  We will show that for each point $p$ in $M$ there are real-valued smooth functions $\phi_1,\ldots,\phi_k$ 
for some $k$ such that the map $(g_1,g_2, \phi_1,\ldots,\phi_k):M\rightarrow \C^2\times \R^k$ is an 
immersion in a neighborhood of $p$ and $\phi_1,\ldots,\phi_k$ are identically zero on $T$.  If $p$ is not 
in $T$, we simply take real-valued functions $x_1,\ldots,x_m$ that form a local coordinate system on a 
neighborhood of $p$ and multiply them by a smooth function that is equal to 1 on a neighborhood of $p$ 
and 0 outside of a closed neighborhood of $p$ contained in the coordinate patch and disjoint from $T$.  
The resulting functions $\phi_1,\ldots,\phi_m$ then have the required properties.  If $p$ is in $T$ 
then there is a local coordinate system $x_1,\ldots, x_m$ on a neighborhood $U$ of $p$ in $M$ such that 
$U\cap T=\{x_1=0,\ldots, x_{m-2}=0\}$.  Then $(g_1,g_2,x_1,\ldots, x_{m-2}):U\rightarrow \C^2\times \R^{m-2}$ 
is an immersion in a neighborhood of $p$.  By multiplying each of $x_1,\ldots,x_{m-2}$ by a smooth function that 
is equal to 1 on a neighborhood of $p$ in $M$ and 0 on a neighborhood of $M\setminus U$, we obtain functions 
$\phi_1,\ldots,\phi_{m-2}$ with the required properties.

By the result of the preceding paragraph and a compactness argument, there exist 
real-valued smooth functions $\phi_1,\ldots,\phi_k$ for some $k$ such that 
$(g_1,g_2,\phi_1,\ldots,\phi_k): M\rightarrow \C^2\times \R\/^k$ is an immersion and 
the functions $\phi_1,\ldots, \phi_k$ are identically 0 on $T$.  Now choose finitely many real-valued 
smooth functions $f_1,\ldots, f_n$ that separate points on $M$.  Finally choose a smooth 
real-valued function $\rho$ on $M$ whose zero-set is exactly $T$.  It is straightforward to verify that the mapping 
$(g_1,g_2,\rho, \rho f_1,\ldots, \rho f_n):M\rightarrow \C^2\times \R^{n+3}$ is injective.  
Thus, the mapping 
$(g_1,g_2,\rho, \rho f_1,\ldots, \rho f_n,\phi_1,\ldots,\phi_k):M\rightarrow \C^2\times \R^{3+n+k}$ 
is an embedding with the desired properties.

{\it Step 2\/}: We show that the $N$ in Step 1 can be taken to be $2m+1$.

The proof is similar to the proof of a version of the Whitney embedding theorem given in \cite{GuilleminPollack}.  Suppose we can show that whenever $N>2m+1$, then there is a nonzero vector $a\in \R\/^N$ such that the composition of $\Phi$ with the orthogonal projection of $\CC\times\RN$ onto the orthogonal complement of $(0,a)$ in $\CC\times \RN$ is an injective immersion (and hence an embedding).  The orthogonal complement of $(0,a)$ in $\CC\times\RN$ is of the form $\CC\times V$ where $V$ is an $(N-1)$-dimensional real vector subspace of $\RN$, so we would get an embedding of $M$ into $\CC\times\R\/^{N-1}$.  Let $p:\CC\times\R\/^N\rightarrow \CC\times\R\/^{N-1}$ denote the map obtained from the orthogonal projection and an identification of $V$ with $\R\/^{N-1}$.  Since the projection is the identity on $\CC\times\{0\}^N$, we get $(p\circ \Phi)(M)\cap (\CC\times\{0\}^{N-1})\supset\Gamma\times\{0\}^{N-1}$.  If in addition $\Phi(M)\cap (\CC\times {\rm span}\{a\})=\Gamma\times\{0\}^N$, then the inclusion is an equality.  Also $(p\circ \Phi)(M)\cap (\CC\times \{r\})$ will contain at most one point for each $r\neq 0$ in $\R\/^{N-1}$ provided $\Phi(M)\cap (\CC\times (\{s\}+{\rm span}\{a\})$ contains at most one point for each $s$ in $\R\/^N\setminus{\rm span}\{a\}$.  So if we can choose $a\in \RN$ so that these conditions are satisfied, then we get an embedding of the required sort into $\CC\times\R\/^{N-1}$.  The existence of the desired embedding into $\CC\times \R\/^{2m+1}$ then follows by induction.

Let $\tphi$ be the composition of $\Phi$ with the map of $\CC\times\RN$ to $\RN$ that projects onto the second factor.  Define a map $h:M\times M\times \R\rightarrow \RN$ by $h(x,y,t)=t[\tphi(x)-\tphi(y)]$.  Also, letting $T(M)$ denote the (real) tangent bundle to $M$, define a map $g:T(M)\rightarrow \RN$ by $g(x,v)=d\tphi_x(v)$.  Since $N>2m+1$, Sard's theorem shows that there exists a point $a$ in $\RN$ belonging to neither the image of $h$ nor the image of $g$.  Note that $a\neq0$.

Let $\pi$ be the projection of $\CC\times \RN$ onto the orthogonal complement $H$ of $(0,a)$.  We want to show that $\pi\circ\Phi:M\rightarrow H$ is injective.  Suppose $(\pi\circ\Phi)(x)=(\pi\circ\Phi)(y)$.  Then $\Phi(x)-\Phi(y)=t(0,a)$ for some scalar $t$, or equivalently $\tphi(x)-\tphi(y)=ta$.  If $x\neq y$, then $t\neq 0$, because $\Phi$ is injective.  But then $h(x,y,1/t)=a$, contradicting the choice of $a$.

Next we want to show  that $\pi\circ\Phi:M\rightarrow H$ is an immersion.  Suppose $v$ is a nonzero vector in $T_x(M)$ for which $d(\pi\circ \Phi)_x(v)=0$.  By the chain rule 
$d(\pi\circ \Phi)_x=\pi\circ d\Phi_x$.  Thus $\pi\circ d\Phi_x(v)=0$,  so $d\Phi_x(v)=t(0,a)$ for some scalar $t$, and $d\tphi_x(v)=ta$.  Since $\Phi$ is an immersion, $t\neq 0$.  Thus $g(x,(1/t)v)=a$, again contradiction the choice of $a$.

Finally we need to consider $\Phi(M)\cap \bigl(\CC\times (\{s\}+{\rm span}\{a\})\bigr)$ for $s\in \RN$.  Because $h(x,y,t)=t[\tphi(x)-\tphi(y)]$ is never equal to $a$, we have that ${\tphi(x)-\tphi(y)}$ is never in ${\rm span}\{a\}$ unless it is zero.  Thus $\tphi(M)$ intersects $\{s\}+{\rm span}\{a\}$ in at most one point for each $s$.  Because $\Phi(M)\cap(\CC\times \{r\})$ contains at most one point for each nonzero $r\in\RN$, this gives that $\Phi(M)\cap \bigl(\CC\times(\{s\}+{\rm span}\{a\})\bigr)$ contains at most one point for $s$ nonzero.  We also get that $\tphi(M)$ intersects ${\rm span}\{a\}$ only in the point $0$ so that $\Phi(M)\cap (\CC\times {\rm span}\{a\})=\Phi(M)\cap(\CC\times\{0\})=\Gamma\times \{0\}$.

{\it Step 3}: We complete the proof of the theorem.

We now have an embedding $\Phi: M\rightarrow\C^2\times \R^{2m+1} \subset \C^{2m+3}$ 
such that $\Phi(M)\cap (\C^2 \times \{0\}^{2m+1}) = \Gamma\times \{0\}^{2m+1}$ and $\Phi(M)\cap (\C^2 \times \{r\})$ 
contains at most one  point for each $r\in \R^{2m+1}$ with $r\neq 0$.  Let $f$ denote the function on the 
standard 2-torus given in Theorem~\ref{graphtorus} whose graph has hull without analytic structure.  
Pull $f$ back to a function on $\Phi^{-1}(\Gamma\times \{0\}^{2m+1})$ by precomposing with $\Phi$, and extend 
the resulting function to a smooth function on $M$ which we will denote by $h$.  
Then $(\Phi,h):M\rightarrow \C^2\times \R^{2m+2}\subset \C^{2m+4}$ is a smooth embedding.  
Let $K=(\Phi,h)(M)$ and $G=(\Phi,h)(\Phi^{-1}(\Gamma\times \{0\}^{2m+2}))$.  Let $K_r=K\cap (\C^2\times \{r\})$.   
Proposition~\ref{proppen} gives that 
$\hat{K}=\bigcup_{r\in\R^{2m+2}}\widehat{K_r}$.  For each $r=(r_1,\ldots, r_{2m+2})\in \C^{2m+2}$ with 
$(r_1,\ldots, r_{2m+1})\neq 0$, we know $K_r$ contains at most one point and hence makes no contribution to 
$\hat{K}\setminus K$.  For $r\in \C^{2m+2}$ with $(r_1,\ldots, r_{2m+1}) =0$ and $r_{2m+2}$ arbitrary, 
$K_r$ is contained in $G$.  It is now easily seen that $\hat{K}\setminus K=\widehat G\setminus G$.  
Since $G$ is the image of the graph of $f$ under the embedding of $\C^3$ into $\C^{2m+4}$ given by 
$(z_1,z_2,z_3)\mapsto (z_1,z_2,0,\ldots, 0, z_3)$, we know that $\widehat G\setminus G$ is non-trivial but
contains no analytic subset of positive dimension by Theorem~\ref{graphtorus}.
\end{proof}

\section{Foliation structure for uniform algebras on manifolds}
Let $M$ be a $\mathcal{C}^1$-smooth manifold of real dimension $n\geq 2$.
Recall that a CR-structure on $M$ is 
a subbundle $\bL$ of the complexified tangent bundle $T^{\C}M$ that is involutive and such that $\bL_p\cap \overline\bL_p=\{0\}$ for each $p\in M$.  
Let $N=\{L+\overline L: L\in \bL\}$.  Then $N$ is a subbundle of the real tangent space $TM$, and there is a complex structure map $J$ on $N$ (i.e., a bundle isomorphism $J:N\rightarrow N$ such that $J^2=-\textrm{Id}$) so that $\bL$ and $\overline \bL$ are respectively the $+i$ and $-i$ eigenspaces of the extension of $J$ to $\bL \oplus \overline \bL$.
We say that a $\mathcal{C}^1$-smooth function $f \colon M\to \C$ is CR, $f\in {\rm CR}(M)$, if $df(Jv)=idf(v)$.

If $f\colon M\to \C^k$ is an embedding and the dimension of
$H_pf(M):=T_p f(M) \cap J_{st} T_p f(M)$,
where $J_{st}$ is the standard complex structure on $\C^k$, is independent of $p\in f(M)$, then there is a CR-structure $\bL$ on $f(M)$ whose associated real subbundle is $H f(M)$.
Moreover, since $Df \colon TM \to T f(M)$ is an isomorphism,
we get an induced CR-structure on $M$ with associated real subbundle 
$N:=(Df)^{-1} H f(M) \subset TM$.

We say that a CR-structure with associated real subbundle $N\subset TM$ is integrable on $M$ if for each point $p\in M$ there is a neighborhood $U\ni p$
and a $\mathcal{C}^1$-smooth mapping $\rho=(\rho_1,\ldots,\rho_{n-2m})\colon U\to \R^{n-2m}$ such that 
$d\rho_1\wedge \cdots \wedge d\rho_{n-2m}$ is non-vanishing in $U$ and such that for every $y\in U$,
the tangent space of $Z_y:=\{x\in U:\, \rho(x)=\rho(y)\}$ equals $N_y$. The local submanifolds $Z_y$
then define a foliation $\mathcal{F}$ on $M$, and $M$ has the structure of a Leviflat
CR-manifold with the Levi foliation $\mathcal{F}$ inducing the given CR-structure.

\begin{proof}[Proof of Theorem~\ref{stoutgen}]
Let $\cA_0$ denote the collection of functions in $\cA$ that are $\cC^1$ on $U$, and let $\tilde\Omega\subset U$ be the set of  
points $x\in U$ such that there exists an open neighborhood $U_x$ of $x$, 
and $f_j\in\mathcal A_0, \jn$, with $f|_{U_x}:U_x\rightarrow\mathbb C^n$ being an 
embedding. We begin by showing that $\tilde{\Omega}$ is dense in $U$; this will depend only on the fact that $\mathcal A_0$ is 
point separating. 

For every $f\in \mathcal{A}_0$ we write $f=u_f+iv_f$, where $u_f$ and $v_f$ are real.
By induction we will pick $f_1,\ldots, f_s,g_1,\ldots,g_t \in \mathcal{A}_0$, $s+t=n$, such that
\begin{equation}\label{wedge}
du_{f_1}\wedge \cdots \wedge du_{f_s}\wedge dv_{g_1}\wedge \cdots\wedge dv_{g_t}
\end{equation} 
is not identically zero: Since $\mathcal{A}_0$ separates points, $\mathcal{A}_0$ must contain a function
whose differential is not identically zero. Assume that we have found $f_1,\ldots, f_s,g_1,\ldots,g_t \in \mathcal{A}_0$, $s+t\leq n$,
such that \eqref{wedge} is non-vanishing on some open set $\Omega_{s+t}$. If $s+t<n$ then some level set of the map $(u_f,v_g)$ defines
a $\cC^1$-smooth submanifold $Y$ of $\Omega_{s+t}$ of positive dimension. Since $\mathcal{A}_0$ is point separating, all functions 
in $\mathcal{A}_0$ cannot be constant on $Y$ and so there is a function in $\mathcal{A}_0$ such that the wedge product of
its differential with \eqref{wedge} is not identically zero.
The resulting map $(f,g)\colon M\to \C^{n}$ now gives an embedding of some neighborhood
of some point in $U$ and hence, $\tilde{\Omega}$ is nonempty.
If $\tilde{\Omega}$ were not dense in $U$, then we could repeat the argument and show that there is a map
in $\mathcal{A}_0^n$ giving an embedding of some neighborhood of some point in $U\setminus \tilde{\Omega}$.

For each $x\in\tilde\Omega$ we let $m_x$ be the smallest integer
such there exists a neighborhood $U_x$ of $x$ and an embedding $f:U_x\rightarrow\mathbb C^k$, $f\in \mathcal{A}_0^k$, for some $k$,  with the property 
that the complex dimension of $H_{f(x)}f(U_x)$ is $m_x$. We claim that there 
exists a dense open subset $\Omega$ of $\tilde\Omega$
such that the map $x\mapsto m_x$ is locally constant on $\Omega$
and strictly greater than zero. To see this
let $m_1=\textrm{min}_{x\in\tilde\Omega} \{m_x\}$. 
Note that if $m_x=0$ then there exists an 
embedding $f:U_x\rightarrow\mathbb C^k$ such that $f(U_x)$ is totally 
real.  Hence $x$ is a local peak point for $\mathcal A$, and so $x$ is in the Shilov boundary. 
Thus, $m_1>0$. Let $\Omega_1:=\{x\in\tilde\Omega: \, m_x=m_1\}$; then $\Omega_1$
is open by upper  semi-continuity of the dimension of the maximal complex tangent space.
If $\Omega_1$ is dense we are done; otherwise let 
$m_2=\textrm{min}_{x\in\tilde\Omega\setminus \overline\Omega_1} \{m_x\}$.
Let $\Omega_2:=\{x\in\tilde\Omega\setminus\overline\Omega_1:m_x=m_2\}$;  
then $\Omega_2$ is open. It is now clear how to proceed to 
obtain $\Omega:=\Omega_1\cup\cdot\cdot\cdot\cup\Omega_s$.

Next we define a CR-structure on $\Omega$ via local embeddings into $\mathbb C^k$
as explained in the beginning of this section.
For $x\in\Omega$ pick an embedding $f:U_x\rightarrow\mathbb C^k$ such that 
$\dim_{\C} H_{f(y)}f(U_x)=m_y=m_x$ for all $y\in U_x$.  Then $f(U_x)$ defines 
a CR-structure ${\rm CR}_f$ on $U_x$.  Let $N_{f}\subset TU_x$ denote the associated real subbundle. 
We claim that if $g\in\mathcal A_0$ then $g\in {\rm CR}_f(U_x)$.
If not, consider the embedding $h=(f,g):U_x\rightarrow \mathbb C^{k+1}$; it then has the property that 
there is a $y\in U_x$ such that $\dim_{\C} H_{h(y)}h(U_x)<m_x=m_y$, which is a contradiction.
It now follows that ${\rm CR}_f$ is in fact independent of the choice of such an embedding $f$. Indeed, if $g$
is another choice of such an embedding at $x$, then, since $g\in {\rm CR}_f(U_x)$, we have that $N_f\subset N_g$, and then, since
$\dim\, N_f=\dim N_g$, it follows that $N_g=N_f$.   
Hence, we have a well defined CR-structure on each component of $\Omega$, and the functions in $\cA_0$ are in ${\rm CR}(\Omega)$.

Finally we show that the obtained CR-structure on $\Omega$ is integrable. 
(Once this is done, it is immediate that every function in $\cA$ is holomorphic along the leaves since the functions in $\cA_0$ are CR and $\cA_0$ is dense in $\cA$.)
By a result of R.~A.~Airapetian \cite{Airapetyan} 
it suffices to show that each
image $f(U_x)\subset \C^k$ of an embedding as above is locally 
polynomially convex.  
We now consider each $g\in\mathcal A$ as a function
on $f(U_x)$. If $U_x$ is sufficiently small it follows from the approximation 
theorem of Salah Baouendi and Fran\c cois Treves\footnote{Although this approximation theorem often is formulated in the $\mathcal{C}^2$-category, 
it holds also in $\mathcal{C}^1$.} that each $g\in\mathcal A$ is uniformly approximable by polynomials on $f(U_x)$.
Thus for each closed neighborhood $V_x$ of $x$ contained in $U_x$, and continuing to regard the functions in 
$\mathcal A$ as functions on $f(U_x)$, we have that $\mathcal A|_{f(V_x)}=[z_1,\ldots, z_k]$.  Consequently $f(V_x)$ 
will be polynomially convex provided the maximal ideal space of $\mathcal A|_{V_x}$ is $V_x$.  
The proof is thus concluded by invoking the following lemma.
\end{proof}

\begin{lemma}\label{aconvex}
Let $A$ be a uniform algebra on a compact Hausdorff space $X$, and suppose that $X$ is the maximal ideal space of $A$.  
Given a point $p\in X$ and a neighborhood $U$ of $p$, there is a closed neighborhood $V$ of $p$ such that the 
maximal ideal space of the restriction algebra $\overline {A|_V}$ is $V$.
\end{lemma}

\begin{proof}
A simple compactness argument shows that there exists a finite collection $\{ f_1,\ldots, f_n \}$ of 
functions in $A$ such that the set $V:=\{\,x\in X: |f_j|\leq 1\quad \forall j=1,\ldots, n \}$ is a (closed) 
neighborhood of $p$ contained in $U$.  Clearly $V$ is $A$-convex in the terminology of \cite[II.6]{Gamelin}.  Thus by 
\cite[II.6.1]{Gamelin}, the maximal ideal space of $\overline{A|_V}$ is $V$.
\end{proof}

For the proof of Theorem~\ref{astructureinhull}, note that taking $\cA=[z_1,\ldots, z_n]_X$ and applying Theorem~\ref{stoutgen} gives the desired dense open set $\Omega$ with a CR-structure on each component.  Since the CR-structure is obtained from local embeddings and was shown to be independent of the choice of embedding, we can now use the canonical global embedding.  Hence the obtained CR-structure makes the components of $\Omega$ into CR-submanifolds of $\C^n$.

\bibliographystyle{amsplain}

\begin{thebibliography}{10}

\bibitem{Airapetyan}
Airapetyan, R. A.; 
Continuation of CR-functions from piecewise-smooth CR-manifolds. 
Mat. Sb. (N.S.) \bf 134\rm, (176) (1987), no. 1, 108--118. (Translation in Math. USSR-Sb. 62 (1989), no. 1, 111--120).

\bibitem{Alexander}
Alexander, H.; 
Hulls of subsets of the torus in $\mathbb C^2$.
\textit{Ann. Inst. Fourier} (Grenoble) \bf 48\rm,  (1998), no. 3, 785--795. 

\bibitem{AI1} Anderson, J. T. and  Izzo, A. J.;  A peak point theorem for
uniform algebras generated by smooth functions on a two-manifold,
\textit{Bull.\ London\ Math.\ Soc.} {\bf 33\/} (2001) 187--195.

\bibitem{AIW}
Anderson, J. T.;  Izzo, A. J.; and Wermer, J.; 
Polynomial approximation on real-analytic varieties in ${\mathbf C\/}^n$\/, 
\textit{Proc.\ Amer.\ Math.\ Soc.} {\bf 132\/} (2004), 1495--1500.


\bibitem{Basener} R. F. Basener,
{\it On rationally convex hulls \/}, Trans.\
Amer.\ Math.\ Soc.\ {\bf 182\/} (1973), 353--381.


\bibitem{Bear} H. S. Bear, {\it Complex function algebras\/}, Trans.\ Amer.\
Math.\ Ann.\ {\bf 90\/} (1959), 383--393.


\bibitem{boggess} A. Boggess, {\it CR Manifolds and the Tangential Cauchy-Riemann Complex\/}, 
Studies in Advanced Mathematics. CRC Press, Boca Raton, FL, 1991.

\bibitem{Cirka}
\v{C}irka, E. M.;
Approximation by holomorphic functions on smooth manifolds in $\mathbb C^n$. 
\textit{Mat. Sb. (N.S.)} \bf  78\rm,  (120) 1969 101--123. 



\bibitem{Demailly}
Demailly, J.-P.; 
Complex Analytic and Differential Geometry.
\textit{On line book, available at:} http://www-fourier.ujf-grenoble.fr/~demailly/manuscripts/agbook.pdf.

\bibitem{DL}
Duval, J. and Levenberg, N.; 
Large polynomially convex hulls with no analytic structure. 
Complex Analysis and Geometry (Trento, 1995), Longman, Harlow, 1997, 119--122.




\bibitem{GuilleminPollack}
Guillemin, V. and Pollack, A.; Differential Topology, Prentice-Hall, Inc., Englewood Cliffs, NJ, 1974.

\bibitem{Hoffman} Hoffman, K.; Banach Spaces of Analytic Functions, Prentice-Hall, Englewood Cliffs, NJ, 1962.


\bibitem{Gamelin} Gamelin, T. W.; 
Uniform Algebras, 2nd ed., Chelsea Publishing Company, New York, NY, 1984.

\bibitem{Izzo93} Izzo, A. J.; Uniform algebras generated by holomorphic and pluriharmonic functions. \textit{Trans.\ Amer.\ Math.\ Soc.} \bf{339}, \rm (1993), 835--847.

\bibitem{Izzo95} Izzo, A. J.; Uniform algebras generated by holomorphic and pluriharmonic functions on strictly pseudoconvex domains. 
\textit{Pacific J. Math.} \bf{171}\rm, (1995), 429--436.

\bibitem{Izzo}
Izzo, A. J.; Failure of Polynomial Approximation on Polynomially Convex 
Subsets of the Sphere. 
\textit{Bull.\ London Math.\ Soc.} \bf 28\rm, (1996), 393--397.


\bibitem{Rosay2006}
Rosay, J.-P.;
Polynomial convexity and Rossi's local maximum principle. 
\textit{Michigan Math. J.} \bf{54}\rm, (2006), no. 2, 427--438.

\bibitem{Spivak}
Spivak, M; 
A Comprehensive Introduction to Differential Geometry, Vol.\ 1.
(2nd edition, Publish of Perish, Wilmington, DE, 1979). 
 
\bibitem{Stout}
Stout, E. L.; 
Polynomial convexity. 
Progress in Mathematics, 261. Birkh\"{a}user Boston, Inc., Boston, MA, 2007. 

\bibitem{Stout1.5}
Stout, E. L.;
Holomorphic approximation on compact, holomorphically convex, real-analytic varieties.
\textit{Proc. Amer. Math. Soc.} \bf 134\rm, (2006) 2302--2308.


\bibitem{Stout2}
Stout, E. L.; 
Algebras on surfaces. 
\textit{Math. Scand.}, to appear.

\bibitem{SW}
Samuelsson, H. and Wold, E. F.; Uniform algebras and approximation on manifolds. 
\textit{Invent. Math.} \bf 188\rm, (2012) 505--523.  

\bibitem{Wermer53}
Wermer, J.;
On algebras of continuous functions.
\textit{Proc. Amer. Math. Soc.} \bf  4\rm, (1953). 866--869. 


\bibitem{Whitney}
Whitney, H.;
Complex analytic varieties. Addison-Wesley Publishing Co. (1972).

\end{thebibliography}

\end{document}